\documentclass[11pt, reqno]{amsart}

\usepackage{amsmath,amssymb,mathtools,amsopn,amsfonts,amsthm}
\usepackage{bbm}
\usepackage{subfigure}
\usepackage{eepic}
\usepackage{a4wide}
\usepackage[utf8]{inputenc}
\usepackage{epsfig}
\usepackage{latexsym}
\usepackage{enumerate, enumitem}
\usepackage[UKenglish]{babel}
\usepackage{verbatim}
\usepackage[initials, non-compressed-cites]{amsrefs}
\usepackage[bookmarks]{hyperref}

\newtheorem{theorem}{Theorem}[section]
\newtheorem{lemma}[theorem]{Lemma}

\newtheorem{claim}[theorem]{Claim}
\theoremstyle{definition}

\numberwithin{equation}{section}

\newcommand{\R}{\mathbb{R}}
\newcommand{\N}{\mathbb{N}}
\newcommand{\C}{\mathbb{C}}
\newcommand{\D}{\mathbb{D}}
\newcommand{\T}{\mathbb{T}}

\renewcommand{\geq}{\geqslant}
\renewcommand{\leq}{\leqslant}
\renewcommand{\epsilon}{\varepsilon}

\setlist[enumerate]{leftmargin=20pt,itemsep=0pt,topsep=0pt}
\setlist[enumerate,1]{label=\emph{(\roman*)},ref={(\roman*)}}

\title[Boundary dynamics of iterated function systems]{A note on the boundary dynamics of holomorphic iterated function systems}
\author{Argyrios Christodoulou}
\address{Department of Mathematics, Aristotle University of Thessaloniki, 54124, Thessaloniki, Greece}
\email{argyriac@math.auth.gr}

\begin{document}

\maketitle

\begin{abstract}
We consider the boundary dynamics of iterated function systems of holomorphic self-maps of the unit disc. Our main result provides a sufficient condition which guarantees that the dynamical behaviour of a left iterated function system in the interior of the unit disc can be extended to the boundary. This generalises an extension of the classical Denjoy--Wolff theorem, due to Bourdon, Matache and Shapiro, to the setting of iterated function systems. To do so, we modify estimates for the Hardy norm of composition operators and combine them with a new technique of perturbing a left iterated function system by elliptic M\"obius transformations.
\end{abstract}

\section{Introduction}
In this article we investigate the boundary dynamics of a sequence of left compositions of functions holomorphic in the unit disc $\D$ of the complex plane.

Like many topics in complex dynamic, the starting point of our analysis is the Denjoy--Wolff theorem, which states that if a holomorphic function $f\colon \D \to\D$ is not conjugate to a Euclidean rotation, then its sequence of iterates $(f^n)$ converges locally uniformly to a constant function $z_0\in\overline{\D}$.

This classical result was strengthened by Boundon, Matache and Shapiro \cite{BoMaSh2005}, and independently by Poggi-Corradini \cite{Po2010}, who showed that in many cases the sequence $(f^n(\zeta))$ converges to $z_0$, for Lebesgue almost all $\zeta$ in the unit circle $\T\vcentcolon=\partial\D$ (see also \cite{DoMa1991}). Here, $f(\zeta)$ is defined in terms of radial limits
\[
f(\zeta)\vcentcolon=\lim_{r\to1^-}f(r\zeta),
\]
which, due to a theorem of Fatou, exist for Lebesgue almost all $\zeta\in\T$.

Recently, there have been many efforts towards generalising the Denjoy--Wolff theorem to include compositions of more than one map \cite{AbSh2025, BEFRS1, BEFRS2, Fe2023, FeNi2025}. Of particular interest are \emph{left iterated function systems}, which are compositions of the form
\[
F_n=f_n\circ f_{n-1} \circ\cdots \circ f_1,
\]
where each $f_i\colon\D\to\D$ is holomorphic. Although there are several sufficient conditions on $f_n$ that guarantee the convergence of the left iterated function system (see, for example, \cite{AbCh2022, AbSh2025, ChSh2019} and \cite[Sections 3.6, 3.7]{Ab2023}), there is no definitive analogue of the Denjoy--Wolff theorem for $F_n$.

Our goal in this article is to bring together the above two approaches of expanding the scope of the Denjoy--Wolff theorem. The theorem of Fatou we mentioned earlier implies that the sequence $F_n(\zeta)$ is well-defined, in terms of radial limits, for Lebesgue almost all $\zeta\in\T$ and all $n\in\N$. Therefore, it is only natural to examine whether the dynamical behaviour of a left iterated function system in $\D$ extends to the boundary. 

Our main result is the following, where $\ell$ denotes the Lebesgue probability measure in $\T$.

\begin{theorem}\label{mainprop}
Let $(f_n)$ be a sequence of holomorphic self-maps of $\D$ such that the left iterated function system $F_n=f_n\circ f_{n-1}\circ\cdots\circ f_1$ converges locally uniformly to a constant $z_0\in\D$. Assume that there exists a subsequence $(f_{n_k})$ of $(f_n)$ and Lebesgue measurable sets $E_k\subseteq\T$ with $\limsup_k \ell(E_k)>0$, so that $\lvert f_{n_k}(\zeta) \rvert<1$ for all $\zeta\in E_k$. Then the sequence $(F_n(\zeta))$ converges to $z_0$ for Lebesgue almost every $\zeta\in\T$.
\end{theorem}

An analysis of the boundary dynamics of holomorphic iterated function systems, with remarkable applications to the theory of transcendental dynamics, was carried out in \cite{BEFRS2} by Benini et al., albeit from a different perspective. To be more precise, \cite[Theorem B]{BEFRS2} provides a sufficient condition for the sequences $(F_n(\zeta))$ and $(F_n(z))$, for $\zeta\in\T$ and $z\in\D$, to exhibit the same asymptotic behaviour. This condition, however, is only substantial when $F_n(z)$ ``escapes" to the boundary $\T$; in particular it fails if one assumes that $F_n$ converges to $z_0\in\D$, as we do in Theorem \ref{mainprop}. Therefore, Theorem \ref{mainprop} can be thought of as an ``interior" (or ``elliptic") companion to the theorem of Benini et al., since together they shed light into the elusive boundary dynamics of left iterated function systems.

Let us now discuss the assumptions of Theorem \ref{mainprop}. Observe that if all the functions $f_n$ in Theorem \ref{mainprop} are \emph{inner}, i.e. $\lvert f_n(\zeta)\rvert =1$ for Lebesgue almost all $\zeta\in\T$, then the sequence $F_n(\zeta)$ lies on $\T$ for all $n$, and so the conclusion of our theorem fails trivially. Similarly, one can easily construct examples which show that our conclusion also fails if the sequence $(f_n)$ contains only finitely many non-inner functions. 

So, it is necessary to assume that $(f_n)$ contains a non-inner subsequence $(f_{n_k})$, as we do in Theorem \ref{mainprop}. Note that $(f_n)$ is still allowed to contain infinitely many inner functions. In Theorem \ref{mainprop}, however, we also assume that the Lebesgue measures of the sets which render $(f_{n_k})$ non-inner do not vanish, as is evident by the positivity of the limit superior in our assumptions. This is to ensure that there is a ``residual" target set which allows the sequence $(F_n(\zeta))$ to transition to the interior of $\D$, where the contracting properties of $f_n$ yield the desired convergence. We expect that an assumption similar to the one proposed in our main result is indeed necessary if one wishes to obtain a similar conclusion.

The techniques we employ, which are inspired by \cite{BoMaSh2005}, involve modifications of classical results for composition operators on Hardy space that appeared in \cite{Sh2000}, in conjunction with a new perturbation technique for iterated function systems. To be more precise, the key idea of our proof is to ``interweave" elliptic M\"obius transformations in-between the compositions of a left iterated function system which render certain Hardy norm estimates valid. 

\section{Preliminaries}

\subsection{Hardy spaces}
We denote by $H^2$ the usual Hardy space consisting of functions $ f(z)=\sum_{n=0}^{\infty}\widehat{f}_nz^n$ holomorphic in $\D$, for which $\sum_{n=0}^{\infty}\rvert \widehat{f}_n\lvert^2<+\infty$. The norm in $H^2$ is defined as
\begin{equation}\label{eq: hardy norm}
\lVert f \rVert_2\vcentcolon=\lim_{r\to1^-}\left(\frac{1}{2\pi}\int_{\T} \lvert f(r\zeta) \rvert^2 d\ell(\zeta)\right)^{1/2}<+\infty,
\end{equation}
where $\ell$ is the Lebesgue probability measure in the unit circle $\T$.

It is known \cite[Theorem 2.2]{Du1970} that for functions $f\in H^2$, the radial limits $f(\zeta)\vcentcolon=\lim_{r\to1}f(r\zeta)$ exist for Lebesgue almost every $\zeta\in\T$. In fact, the norm in $H^2$ can be obtained by the boundary integral
\begin{equation}\label{eq: hardy norm integral form}
\lVert f \rVert_2=\left(\frac{1}{2\pi}\int_{\T} \lvert f(\zeta) \rvert^2d\ell(\zeta)\right)^{1/2}.
\end{equation}
Another way of computing this norm involving an area integral is given by the Littlewood--Paley Identity \cite[p. 178]{Sh1993}
\begin{equation}\label{L-P}
\lVert f \rVert_2^2 = \lvert f(0) \rvert^2 +2\int_{\D} \lvert f'(z) \rvert^2 \log \frac{1}{\lvert z\rvert} dA(z),
\end{equation}
where $A$ is the Lebesgue probability measure in $\D$.

\subsection{Composition operators}
For a holomorphic function $f\colon\D\to\D$, we define the linear \emph{composition operator} $C_f\colon H^2\to H^2$ with $C_f(g)=g\circ f$.

Littlewodd's Subordination Principle \cite[Theorem 1.7]{Du1970} shows that this operator is well-defined and bounded. In particular, we have $\lVert C_f \rVert \leq \sqrt{\frac{1+\lvert f(0) \rvert}{1-\lvert f(0) \rvert}}$ (see \cite[p. 16]{Sh1993}). Note that in the special case where $f(0)=0$ this inequality becomes 
\begin{equation}\label{eq: littlewood norm}
\lVert C_f \rVert \leq 1.
\end{equation}

At this point we need to make a remark about the notation of norms. The notation $\lVert C_f \rVert$ used above denotes the operator norm induced by the $H^2$ norm. We shall follow this convention throughout the article; i.e. $\lVert \cdot \rVert_2$ will denote the norm of a function in $H^2$, whereas $\lVert \cdot \rVert$ will denote the operator norm of a composition operator.

The main tool for our proof of Theorem \ref{mainprop} is evaluating the norms of composition operators in $H^2$. In order to do so, we require the so-called ``Change of variables formula" that appears in \cite[p. 179]{Sh1993}. To present this, let us first define the Nevanlinna counting function of a holomorphic function $f\colon\D\to\D$, as 
\[
N_f(w)=\begin{cases}\displaystyle \sum_{z\in f^{-1}(\{w\})}\log \frac{1}{\lvert z \rvert}, \quad \text{for all}, & w\in f(\D)\setminus\{f(0)\},\\
					0, &w\notin f(\D)\setminus\{f(0)\}. \end{cases}
\]
where points in the preimage $f^{-1}(\{w\})$ are repeated according to their multiplicity. The change of variables formula we promised is the following:
\begin{equation}\label{change_var}
\lVert C_f(g) \rVert_2^2 = \lvert g(f(0)) \rvert^2 + 2 \int_{\D}\lvert g'(w)\rvert^2 N_f(w) dA(w).
\end{equation}

We now present some properties of the Nevanlinna counting function that will be important in our proofs. First, an application of Jensen's formula yields an integral interpretation of the Nevanlinna counting function, as follows (see \cite[p. 187]{Sh1993}):
\[
N_f(w) = \lim_{r\to1^-}\int_{\T}\log\left\lvert \frac{w-f(r\zeta)}{1-\overline{w}f(r\zeta)}\right\rvert d\ell(\zeta) + \log\left\lvert\frac{1-\overline{w}f(0)}{w-f(0)}\right\rvert.
\]
Since functions in $H^2$ have radial limits almost everywhere in $\T$, we can use Fatou's Lemma to obtain that whenever $f\in H^2$, we have
\begin{equation}\label{fatoulemma}
N_f(w) \leq \int_{\T}\log\left\lvert \frac{w-f(\zeta)}{1-\overline{w}f(\zeta)}\right\rvert d\ell(\zeta) + \log\left\lvert\frac{1-\overline{w}f(0)}{w-f(0)}\right\rvert.
\end{equation}
Another estimate of the Nevanlinna counting function is the famous Littlewood's Inequality \cite[p.187]{Sh1993}, which states that for any $f\colon \D\to\D$ holomorphic and any $w\in\D\setminus \{f(0)\}$, we have
\begin{equation}\label{eq: littlewood ineq}
N_f(w)\leq \log\left\lvert \frac{ 1-\overline{w}f(0)}{w-f(0)}\right\rvert. 
\end{equation}
In fact, \eqref{fatoulemma} is simply an intermediate step in the proof of Littlewood's Inequality in \cite[p.187]{Sh1993}.

\subsection{Hyperbolic geometry} We finish the preliminary material by presenting some basic facts for the hyperbolic geometry of the unit disc. For a complete treatise on hyperbolic geometry, we refer to \cite[Chapter 1]{Ab2023}.

To start with, the \emph{hyperbolic metric} $d$ of $\D$ is defined as 
\[
d(z,w)=\frac{1}{2}\log \frac{1+\left\lvert\frac{z-w}{1-\overline{w}z}\right\rvert}{1-\left\lvert\frac{z-w}{1-\overline{w}z}\right\rvert}, \quad \text{ for all}\quad z,w\in\D.
\]
It is easy to see that $d$ is indeed a metric, which turns $\D$ into a complete metric space. The geodesics of this metric space are arcs of circles orthogonal to the unit circe $\T$, and the space in uniquely geodesic; i.e. any two points $z,w\in\D$ can be joined by a unique geodesic of $(\D,d)$. 

Moreover, if $\gamma\colon[0,1]\to\D$ is a geodesic joining $z,w\in\D$, then there exists a unique point $a\in\gamma([0,1])$ so that 
$d(z,a)=d(w,a)=d(z,w)/2$. We call this point $a$ the \emph{hyperbolic midpoint of $z$ and $w$}.

The main advantage of working with the hyperbolic metric of $\D$, instead of the Euclidean, is the Schwarz--Pick Lemma, which can be stated as follows:\\
If $f\colon\D\to\D$ is holomorphic, then 
\begin{equation}\label{eq: schwarz-pick}
d(f(z),f(w))\leq d(z,w),\quad\text{for all}\quad z,w\in \D.
\end{equation}
Equality in \eqref{eq: schwarz-pick} is achieved for some (and hence any) $z\neq w$ if and only if 
\[
f(z)=e^{i\theta}\frac{z-a}{1-\overline{a}z}, \quad \text{for some}\ \theta\in\R \ \text{and}\ a\in\D.
\]
The M\"obius transformations of this form will be called \emph{automorphisms of $\D$}, and are exactly the orientation preserving isometries of $d$.

\section{Proof of the main result}
This section is dedicated to proving Theorem~\ref{mainprop}. For the convenience of the reader, let us recall the assumptions of Theorem~\ref{mainprop}. Suppose $(f_n)$ is a sequence of holomorphic self-maps of $\D$ such that the left iterated function system $F_n=f_n\circ f_{n-1}\circ \cdots\circ f_1$ converges locally uniformly to a constant $z_0\in\D$. We assume that there exists a subsequence $(f_{n_k})\subseteq (f_n)$ and $\ell$-measurable sets $E_k\subset \T$ such that $\lvert f_{n_k}(\zeta)\rvert<1$ for all $\zeta\in E_k$ and $\limsup_k \ell(E_k)>0$. 

Our goal is to show that $F_n(\zeta)\to z_0$, as $n\to+\infty$, for $\ell$-almost all $\zeta\in\T$. Observe that if $f_k$ is constant, for some $k\in\N$, then the result is trivial due to the convergence of $(F_n)$ to $z_0$. So, for the rest of the proof, we assume that all $f_n$ are non-constant self-maps of $\D$. Also, consider the automorphism 
\[
\psi_0(z)=\frac{z-z_0}{1-\overline{z_0}z},
\]
and the functions $g_n=\psi_0 \circ f_n\circ \psi_0^{-1}$. We can see that $(F_n(\zeta))$ converges to $z_0$ if and only if the sequence $G_n(\zeta)=g_n\circ\cdots\circ g_1(\zeta)=\psi_0\circ F_n\circ \psi_0^{-1}(\zeta)$ converges to 0. Therefore, it suffices to assume that $z_0=0$.

Throughout the proof we use $D(z,r)$ to denote the Euclidean disc of radius $r>0$, centred at $z\in\C$. Before we proceed with the main body of the proof, we require the following result from \cite[Theorem 3.3]{ChSh2019} (see, also, \cite[Theorem 4.1 (ii)]{AbCh2022}).

\begin{theorem}\label{thm: stability}
Let $f$ be a holomorphic self-map of $\D$ with $f(z_0)=z_0$, for some $z_0\in\D$, which is not an automorphism of $\D$. If $(f_n)$ is a sequence of holomorphic self-maps of $\D$ that converges locally uniformly to $f$, then the left iterated function system $F_n=f_n\circ f_{n-1}\circ\cdots \circ f_1$ converges locally uniformly to the constant $z_0$.
\end{theorem}

The main idea for our proof is to find an appropriate normalisation for the compositions $F_n$ that will allow us to employ estimates on certain Hardy norms. Such a normalisation can be obtained as follows:

By passing to a subsequence of $(f_{n_k})$, we can assume that there exists $c\in(0,1)$ such that $\ell(E_k)> c$, for all $k\in\N$. Moreover, for every $\epsilon>0$, and all $k\in\N$, there exists a Lebesgue measurable set $E_k'\subseteq E_k$ and a constant $\delta_{k,\epsilon}\in(0,1)$ so that $\lvert f_{n_k}(\zeta) \rvert<\delta_{k,\epsilon}$, for every $\zeta\in E_k'$, and $\ell(E_k')\geq \ell(E_k)-\epsilon$. So, without loss of generality, we also assume that for every $k\in\N$, there exists $\delta_k\in(0,1)$, so that 
\begin{equation}\label{dk}
\lvert f_{n_k}(\zeta) \rvert<\delta_k,\quad  \text{for every} \quad \zeta\in E_k
\end{equation}
and $\ell(E_k)> c$, for some $c\in(0,1)$.

Next, we define the compositions $\phi_k=f_{n_{k+1}-1}\circ\cdots\circ f_{n_k}$ for $k\geq 1$. Then, if $m\in\N$ and $k$ is the largest integer such that $m>n_{k+1}$, we have that 
\[
F_m= f_m\circ\cdots\circ f_{n_{k+1}}\circ \phi_k\circ\phi_{k-1}\circ\cdots\circ \phi_1\circ f_{n_1-1}\circ\cdots \circ f_1, \quad \text{for all}\ m \ \text{large enough}.
\]

We would first like to examine the convergence properties of $(\phi_k)$. Since $(F_n)$ converges locally uniformly to 0 and all $f_n$ are non-constant (as assumed in the beginning of this section), for any fixed $n_0\in\N$, the sequence $(f_n\circ\cdots\circ f_{n_0})$ converges to 0, pointwise on the open set $f_{n_0-1}\circ\cdots \circ f_1\left(D(0,\tfrac{1}{2})\right)$. Then, the Vitali--Porter Theorem \cite[p.75]{schiff} implies that $(f_n\circ\cdots\circ f_{n_0})$ converges to 0, locally uniformly as $n\to+\infty$. So, by passing to a further subsequence if necessary, we can assume that the integers $n_k$ are sparse enough, so that for all $k\in\N$, the following hold:
\begin{align}
&\phi_k\left(\overline{D\left(0,\frac{1}{2}\right)}\right)\subset D\left(0,\frac{1}{2}\right),\label{assumption 1} \\
&\phi_k(z)\xrightarrow{k\to\infty}0,\quad\text{for all}\ z\in\D,\quad  \text{and}\label{assumption 2}\\
&\lvert \phi_k(\zeta) \rvert<\frac{1}{2}, \quad \text{for all} \ \zeta\in E_k\label{assumption 3}.
\end{align}
This last inequality, \eqref{assumption 3}, is possible due to \eqref{dk}.

The sequence $(\phi_k)$ is the key ingredient of our normalisation scheme. However, our work is not done, as we would like each $\phi_k$ to fix the point 0, which might not happen in principle. 

To overcome this obstacle, we consider the automorphism $e_1(z)=\frac{\phi_1(0)-z}{1-\overline{\phi_1(0)}z}$ and recursively define the automorphism $e_k$ with
\[
e_k(z)=\frac{\phi_k(e_{k-1}(0))-z}{1-\overline{\phi_k(e_{k-1}(0))}z}, \quad\text{for}\quad k\geq2.
\]
Let us examine some properties of the automorphisms $e_k$. Firstly, observe that each $e_k$ is self-inverse. Moreover, $e_k$ interchanges the points 0 and $\phi_k(e_{k-1}(0))$ (or 0 and $\phi_1(0)$ for the case of $e_1$). Therefore, every $e_k$ has a fixed point $w_k\in\D$, which is the hyperbolic midpoint
of 0 and $\phi_k(e_{k-1}(0))$ (and analogously for $e_1$). Now, observe that 
\begin{equation}\label{eq: ek in disc}
e_k(0)=\phi_k(e_{k-1}(0))=\phi_k(\phi_{k-1}(e_{k-2}(0)))=\cdots=\phi_k\circ\cdots\circ \phi_1(0)\in D\left(0,\frac{1}{2}\right),
\end{equation}
due to \eqref{assumption 1}. Therefore, we have that
\begin{equation}\label{eq: fixed points of ek}
w_k\in D\left(0,\frac{1}{2}\right), \quad \text{for all}\ k\in\N.
\end{equation}

We can in fact say more about the sequence of points $e_k(0)=\phi_k\circ\cdots\circ\phi_1(0)$. Note that \eqref{assumption 1} and \eqref{assumption 2} imply that $(\phi_k)$ converges locally uniformly to the constant function 0 (again by the Vitali--Porter Theorem). So, we can use Theorem \ref{thm: stability} in order to deduce that the left iterated function system $\Phi_k=\phi_k\circ \phi_{k-1}\circ\cdots \circ \phi_1$ converges locally uniformly to 0. This means that
\begin{equation}\label{eq: ek converges to 0}
e_k(0)=\Phi_k(0)\xrightarrow{k\to+\infty}0.
\end{equation}

Next, we define the sequence of functions $\widetilde{\phi}_1=e_1\circ\phi_1$ and $\widetilde{\phi}_k=e_k\circ \phi_k\circ e_{k-1}$, for $k\geq 2$. Notice that $\widetilde{\phi}_k(0)=0$ for all $k\geq1$, as required for our technique. Also, since each $e_k$ is self-inverse, we can rewrite the composition $F_m$ as follows,
\begin{align}
F_m&=f_m\circ \cdots \circ f_{n_{k+1}}\circ e_k\circ e_k \circ \phi_k \circ e_{k-1}\circ e_{k-1}\circ \phi_{k-1}\circ\cdots \circ e_1\circ \phi_1\circ f_{n_1-1}\circ\cdots \circ f_1\nonumber\\
	&= f_m\circ\cdots \circ f_{n_{k+1}}\circ e_k\circ \widetilde{\phi}_k\circ \widetilde{\phi}_{k-1}\circ \cdots\circ \widetilde{\phi}_1\circ f_{n_1-1}\circ\cdots \circ f_1,\label{Fm}
\end{align}
which is valid for all $m$ large. 

Rewriting the left iterated function system $(F_n)$ in the form \eqref{Fm} will turn out to be exactly the normalisation needed for our purposes. So, with this endeavour completed, we move on to obtaining estimates on the norms of the composition operators induced by the functions $\widetilde{\phi_k}$.

Observe that \eqref{assumption 3} is equivalent to $\lvert \phi_k\circ e_{k-1}(\zeta) \rvert <\frac{1}{2}$, for all $\zeta\in e_{k-1}(E_k)$. Since each $e_k$ is an automorphism of $\D$ with fixed point $w_k\in D\left(0,\frac{1}{2}\right)$ (see \eqref{eq: fixed points of ek}), there exists a universal constant $r_0\in(0,1)$, so that $e_k\left(D\left(0,\frac{1}{2}\right)\right)\subset D(0,r_0)$, for all $k\in\N$. Therefore, 
\begin{equation}
\left\lvert\widetilde{\phi}_k(\zeta)\right\rvert<r_0,\quad \text{for all}\quad \zeta\in e_{k-1}(E_k)\quad \text{and every}\quad k\in\N.\label{innerset}
\end{equation}

Now, simple computations show that for all $\zeta\in\T$, we have
\begin{equation}\label{eq: derivative estimate}
\lvert e_k'(\zeta)\rvert = \frac{1-\lvert \phi_k(e_{k-1}(0))\rvert^2}{\left\lvert 1-\overline{\phi_k(e_{k-1}(0))}\zeta\right\rvert^2}\geq \frac{3}{16},
\end{equation}
where to obtain the inequality we used the triangle inequality and \eqref{eq: ek in disc}. Since each $e_k$ is a conformal automorphism of $\D$, \cite[Theorem 6.8]{Po1992} implies that 
\[
\ell(e_{k-1}(E_k))=\int_{E_k}\lvert e_k'(\zeta)\rvert d\ell(\zeta),
\]
which, when combined with \eqref{eq: derivative estimate} yields
\begin{equation}\label{measure}
\ell\left(e_{k-1}(E_k)\right)\geq \frac{3}{16}\ell(E_k)\geq \frac{3c}{16}. 
\end{equation}

Our first claim is the analogue of \cite[Theorem 3.2]{Sh2000} to our setting.

\begin{claim}\label{claim1}
There exist constants $r_1,\delta\in(0,1)$ so that for all $k\in\N$
\[
N_{\widetilde{\phi}_k}(w)\leq \delta\ \log\frac{1}{\lvert w \rvert}, \quad \text{for all} \quad r_1\leq \lvert w \rvert< 1.
\]
\end{claim}

\begin{proof}[Proof of Claim~\ref{claim1}]
Let $w\in\D$, $k\in\N$ and $\zeta\in e_{k-1}(E_k)$. Recall that by inequality \eqref{innerset} we have that $\lvert\widetilde{\phi}_k(\zeta)\rvert <r_0$. So,
\begin{align}
1-\left\lvert \frac{w-\widetilde{\phi}_k(\zeta)}{1-\overline{w}\widetilde{\phi}_k(\zeta)} \right\rvert^2&=\frac{(1-\lvert w \rvert^2)(1-\lvert \widetilde{\phi}_k(\zeta) \rvert^2)}{\lvert 1- \overline{w}\widetilde{\phi}_k(\zeta)\rvert^2}\geq \frac{(1-\lvert w \rvert^2)(1-{r_0}^2)}{(1+\lvert\widetilde{\phi}_k(\zeta)\rvert)^2}\nonumber\\
	&\geq \frac{(1-\lvert w \rvert^2)(1-{r_0}^2)}{(1+r_0)^2}=\frac{(1-\lvert w \rvert^2)(1-{r_0})}{1+r_0}. \label{eq: claim1 eq1}
\end{align}
Moreover, $1-x\leq \log\frac{1}{x}$, for all $x\in(0,1)$, and there exists $r'\in(0,1)$ so that $\log\frac{1}{x}\leq 2(1-x)$, for all $x\in[r',1)$. Write $r_1\vcentcolon=\sqrt{r'}$. Then, using these inequalities and \eqref{eq: claim1 eq1}, we get that for all $r_1\leq\lvert w \rvert <1$,
\begin{align}
\log \left\lvert \frac{1-\overline{w}\widetilde{\phi}_k(\zeta)}{w-\widetilde{\phi}_k(\zeta)} \right\rvert^2&\geq 1-\left\lvert \frac{w-\widetilde{\phi}_k(\zeta)}{1-\overline{w}\widetilde{\phi}_k(\zeta)} \right\rvert^2\geq \frac{(1-\lvert w \rvert^2)(1-{r_0})}{1+r_0}\nonumber\\
	&\geq \frac{1-{r_0}}{2(1+r_0)}\log\frac{1}{\lvert w \rvert^2}=\frac{1-{r_0}}{1+r_0}\log\frac{1}{\lvert w \rvert}\label{claimineq}.
\end{align}
Observe that 
\[
\left\lvert \frac{w-\widetilde{\phi}_k(\zeta)}{1-\overline{w}\widetilde{\phi}_k(\zeta)} \right\rvert\leq 1,
\]
and so
\begin{equation}\label{eq: claim log negative}
\log\left\lvert \frac{w-\widetilde{\phi}_k(\zeta)}{1-\overline{w}\widetilde{\phi}_k(\zeta)} \right\rvert\leq 0, \quad\text{for all}\ \zeta\in e_{k-1}(E_k).
\end{equation}
Using inequality \eqref{eq: claim log negative} and the inequality \eqref{fatoulemma} for the Nevanlinna counting function, we obtain that for all $w\in\D$ and all $k\in\N$,
\[
N_{\widetilde{\phi}_k}(w)\leq \int_{\T}\log\left\lvert \frac{w-\widetilde{\phi}_k(\zeta)}{1-\overline{w}\widetilde{\phi}_k(\zeta)} \right\rvert d\ell(\zeta)+\log\frac{1}{\lvert w \rvert}\leq \int_{e_{k-1}(E_k)} \log\left\lvert \frac{w-\widetilde{\phi}_k(\zeta)}{1-\overline{w}\widetilde{\phi}_k(\zeta)} \right\rvert d\ell(\zeta)+\log\frac{1}{\lvert w \rvert}.
\]
Therefore, restricting to $r_1\leq \lvert w\rvert <1$ and using \eqref{claimineq} yields
\[
N_{\widetilde{\phi}_k}(w)\leq -\frac{1}{2}\frac{1-{r_0}}{1+r_0}\ell(e_{k-1}(E_k))\log\frac{1}{\lvert w \rvert} + \log\frac{1}{\lvert w \rvert} = \left(1-\frac{1}{2}\frac{1-{r_0}}{1+r_0}\ell(e_{k-1}(E_k))\right) \log\frac{1}{\lvert w \rvert}.
\]
Finally, using the estimate $\ell(e_{k-1}(E_k))\geq \frac{3c}{16}$ from \eqref{measure}, we obtain
\[
N_{\widetilde{\phi}_k}(w)\leq \left(1-\frac{3c(1-{r_0})}{32(1+r_0)}\right)\log\frac{1}{\lvert w \rvert},
\]
which is the desired conclusion for $\delta = 1-\frac{3c(1-{r_0})}{32(1+r_0)}\in (0,1)$.
\end{proof}

As promised, we are now ready to prove an upper bound for the norm of the composition operator $C_{\widetilde{\phi}_k}$, similar to \cite[Theorem 5.1]{Sh2000}. Let us denote by $H_0^2$ the linear subspace of $H^2$ consisting of all functions $f\in H^2$ with $f(0)=0$. Since $\widetilde{\phi}_k(0)=0$, the restriction $C_{\widetilde{\phi}_k}\arrowvert_{H_0^2}$ is a well-defined operator on $H_0^2$. 

\begin{claim}\label{claim2}
There exists a constant $\nu\in(0,1)$, such that $\lVert C_{\widetilde{\phi}_k}\arrowvert_{H_0^2}\rVert\leq \nu$, for all $k\in\N$.
\end{claim}

For the proof of Claim~\ref{claim2}, we require the following simple estimate from measure theory (for a proof, see \cite[Lemma 2.7]{Sh2000}).

\begin{lemma}\label{shlemma}
Suppose $\mu$ is a positive, finite Borel measure on $[0,1)$ that has the point 1 is in its closed support. For every $r\in(0,1)$, there exists a constant $\gamma\in(0,1)$, depending only on $r$ and $\mu$, such that for every non-negative, increasing function $g$ on $[0,1)$, we have
\[
\int_{[0,r)}g d \mu\leq \gamma \int_{[0,1)} g d \mu.
\]
\end{lemma}

\begin{proof}[Proof of Claim~\ref{claim2}]
Suppose $g\in H_0^2$. By the change of variables formula \eqref{change_var}, we have that
\[
\lVert C_{\widetilde{\phi}_k}(g) \rVert_2^2 = 2\int_{\D} \lvert g'(w) \rvert N_{\widetilde{\phi}_k}(w)d A(w).
\]
From Claim~\ref{claim1}, there exist constants $r_1,\delta\in(0,1)$ so that for all $w\in \D$, with $\lvert w \rvert\in[r_1,1)$, we have $N_{\widetilde{\phi}_k}(w)\leq \delta\log\frac{1}{\lvert w \rvert}$. Also, Littlewood's Inequality \eqref{eq: littlewood ineq} applied to $\widetilde{\phi}_k$ implies that $N_{\widetilde{\phi}_k}(w)\leq \log\frac{1}{\lvert w\rvert}$, for all $w\in\D\setminus\{\widetilde{\phi}_k(0)\}$. So,
\begin{align}
\lVert C_{\widetilde{\phi}_k}(g) \rVert_2^2 &= 2 \int_{\lvert w \rvert< r_1}\lvert g'(w)\rvert^2 N_{\widetilde{\phi}_k}(w)dA(w) + 2\int_{\lvert w \rvert\geq r_1}\lvert g'(w)\rvert^2 N_{\widetilde{\phi}_k}(w)dA(w)\nonumber \\
	&\leq 2 \int_{\lvert w \rvert< r_1}\lvert g'(w)\rvert^2 \log\frac{1}{\lvert w \rvert} dA(w) + 2\delta \int_{\lvert w \rvert\geq r_1}\lvert g'(w)\rvert^2 \log\frac{1}{\lvert w \rvert}dA(w)\nonumber\\
	&=2(1-\delta) \int_{\lvert w \rvert< r_1}\lvert g'(w)\rvert^2 \log\frac{1}{\lvert w \rvert} dA(w) + 2\delta \int_{\D}\lvert g'(w)\rvert^2 \log\frac{1}{\lvert w \rvert}dA(w).\label{int_break}
\end{align}
Now, define $h(r)=\frac{1}{2\pi}\int_0^{2\pi}\lvert g'(re^{i\theta})\rvert^2 d\theta$, and note that it is an increasing function of $r\in[0,1)$. Moreover, if we consider the Borel measure $d\mu(t)=2t\log\frac{1}{t}dt$, then Lemma~\ref{shlemma} yields a constant $\gamma\in(0,1)$, depending only on $r_1$, so that
\[
\int_{[0,r_1)}h(r)d\mu(r)=\int_{\lvert w \rvert<r_1}\lvert g'(w) \rvert^2\log\frac{1}{\lvert w \rvert}d A(w)\leq \gamma \int_{\D} \lvert g'(w) \rvert^2\log\frac{1}{\lvert w \rvert}d A(w).
\]
Therefore, \eqref{int_break} implies that
\begin{equation}\label{eq: norm lemma eq1}
\lVert C_{\widetilde{\phi}_k}(g) \rVert_2^2 \leq (\gamma(1-\delta)+\delta) 2\int_{\D} \lvert g'(w) \rvert^2\log\frac{1}{\lvert w \rvert}d A(w).
\end{equation}
By the Littlewood--Paley Identity \eqref{L-P}, the inequality \eqref{eq: norm lemma eq1} can be rewritten as 
\[
\lVert C_{\widetilde{\phi}_k}(g) \rVert_2^2 \leq (\gamma(1-\delta)+\delta)\lVert g \rVert_2^2.
\]
Since $g\in H_0^2$ was arbitrary, we have that $\lVert C_{\widetilde{\phi}_k}\arrowvert_{H_0^2} \rVert\leq \sqrt{\gamma(1-\delta)+\delta} <1$, as required.
\end{proof}

Now that the groundwork is laid, we can complete the proof of our main result.

\begin{proof}[Proof of Theorem~\ref{mainprop}]
Since, as mentioned in the beginning of this section, we assume that $z_0=0$, our goal is to show that the sequence $(F_m(\zeta))$ converges to 0, as $m\to+\infty$, for Lebesgue almost all $\zeta\in\T$.

Recall that from our normalisation \eqref{Fm}, the left iterated function system $F_m$ can be rewritten as
\[
F_m= f_m\circ\cdots \circ f_{n_{k+1}}\circ e_k\circ \widetilde{\phi}_k\circ \widetilde{\phi}_{k-1}\circ \cdots\circ \widetilde{\phi}_1\circ f_{n_1-1}\circ\cdots \circ f_1,
\]
where $k$ is the largest integer so that $m>n_{k+1}$. Therefore, 
\[
F_{n_{k+1}-1}= e_k\circ \widetilde{\phi}_k\circ \widetilde{\phi}_{k-1}\circ \cdots\circ \widetilde{\phi}_1\circ f_{n_1-1}\circ\cdots \circ f_1.
\]
For simplicity, let us also write $f=f_{n_1-1}\circ\cdots \circ f_1$, so that
\[
e_k\circ F_{n_{k+1}-1}= \widetilde{\phi}_k\circ \widetilde{\phi}_{k-1}\circ \cdots\circ \widetilde{\phi}_1\circ f,
\]
recalling that the automorphisms $e_k$ are self-inverse. 

Our first goal is to show that $(e_k\circ F_{n_{k+1}-1}(\zeta))$ converges to 0 for Lebesgue almost all $\zeta\in\T$.

Using Claim~\ref{claim2}, we can find $\nu\in(0,1)$ so that $\lVert C_{\widetilde{\phi}_k}\arrowvert_{H_0^2}\rVert\leq \nu$, for all $k\in\N$. Therefore, for the $H^2$ norm of $e_k\circ F_{n_{k+1}-1}$ we get 
\begin{align}
\lVert e_k\circ F_{n_{k+1}-1} \rVert_2 &= \lVert C_{e_k\circ F_{n_{k+1}-1}}(\mathrm{id})\rVert_2 = \lVert C_f\circ C_{\widetilde{\phi}_1}\circ C_{\widetilde{\phi}_2}\circ \cdots \circ C_{\widetilde{\phi}_k}(\mathrm{id})\rVert_2\nonumber\\
	&\leq \lVert C_f \rVert \cdot \lVert C_{\widetilde{\phi}_1}\circ C_{\widetilde{\phi}_2}\circ \cdots \circ C_{\widetilde{\phi}_k}(\mathrm{id})\rVert_2\nonumber\\
	&\leq \lVert C_f \rVert \cdot \prod_{i=1}^{k}\lVert C_{\widetilde{\phi}_i}\arrowvert_{H_0^2}\rVert\nonumber\\
	&\leq\lVert C_f \rVert\cdot \nu^k.\label{eq: gmFm norm}
\end{align}

Using the Monotone Convergence Theorem, the boundary integral form of the $H^2$ norm in \eqref{eq: hardy norm integral form}, and \eqref{eq: gmFm norm} we obtain
\begin{align*}
\int_{\T} \sum_{k=1}^{+\infty} \lvert  e_k\circ F_{n_{k+1}-1}(\zeta)\rvert^2 d\ell(\zeta) &= \sum_{k=1}^{+\infty} \int_{\T} \lvert  e_k\circ F_{n_{k+1}-1}(\zeta)\rvert^2 d\ell(\zeta) \\
&=\sum_{k=1}^{+\infty} \lVert  e_k\circ F_{n_{k+1}-1} \rVert^2 \leq \lVert C_f \rVert^2\cdot \sum_{k=1}^{+\infty} \nu^{2k}<+\infty.
\end{align*}

This means that $\sum\limits_{k=1}^{+\infty} \lvert e_k\circ F_{n_{k+1}-1}(\zeta)\rvert^2<+\infty$, for almost every $\zeta\in\T$, and so 
\begin{equation}\label{eq: final convergence}
\lvert e_k\circ F_{n_{k+1}-1}(\zeta) \rvert\xrightarrow{k\to+\infty} 0, \quad \text{for almost every}\  \zeta\in\T,
\end{equation} 
as desired.

We finally move on to the general case where
\begin{align*}
F_m&= f_m\circ\cdots \circ f_{n_{k+1}}\circ e_k\circ \widetilde{\phi}_k\circ \widetilde{\phi}_{k-1}\circ \cdots\circ \widetilde{\phi}_1\circ f_{n_1-1}\circ\cdots \circ f_1\\
&=f_m\circ\cdots \circ f_{n_{k+1}}\circ e_k\circ F_{n_{k+1}-1},
\end{align*}
and $k$ is the largest integer so that $m>n_{k+1}$. In particular, $k\to+\infty$ whenever $m\to+\infty$.

Note that \eqref{eq: final convergence} implies that $e_k\circ F_{n_{k+1}-1}(\zeta)\in\D$, for all $k$ large, and so $F_m(\zeta)$ also lies in $\D$ for all $m$ large.

By the triangle inequality for the hyperbolic metric and the Schwarz--Pick Lemma \eqref{eq: schwarz-pick}, we have
\begin{align}
d(f_m\circ\cdots\circ f_{n_{k+1}}(0),0)&\leq d(f_m\circ\cdots\circ f_{n_{k+1}} (0),F_m(0))+ d(F_m(0),0)\nonumber\\
		&\leq d(0,F_{n_{k+1}-1}(0)) + d(F_m(0),0).\label{eq: tail}
\end{align}
Now, since by assumption $F_m$ converges locally uniformly to the constant 0, we have that the sequences $(F_{n_{k+1}-1}(0)$ and $(F_m(0))$ both converge to 0 as $m$---and so also as $k$---tends to $+\infty$. So, taking limits $m\to+\infty$ in \eqref{eq: tail} and applying these facts yields
\begin{equation}\label{eq: tail converges}
f_m\circ\cdots\circ f_{n_{k+1}}(0)\xrightarrow{m\to+\infty}0.
\end{equation}
To conclude the proof it suffices to note that for almost all $\zeta\in\D$ we have
\begin{align*}
d(F_m(\zeta),f_m\circ\cdots\circ f_{n_{k+1}}(0))\leq d(e_k\circ F_{n_{k+1}-1}(\zeta),0),
\end{align*}
and apply \eqref{eq: final convergence} and \eqref{eq: tail converges}.
\end{proof}

\section*{Acknowledgements}
I would like to thank Professors M. D. Contreras and S. D\'iaz-Madrigal for the discussions that inspired this result, and the anonymous referee for their insightful comments and recommendations.


\begin{bibdiv}
\begin{biblist}

\bib{Ab2023}{book}{
   author={Abate, Marco},
   title={Holomorphic dynamics on hyperbolic Riemann surfaces},
   series={De Gruyter Studies in Mathematics},
   volume={89},
   publisher={De Gruyter, Berlin},
   date={2023},
   pages={xiii+356}
}

\bib{AbCh2022}{article}{
   author={Abate, Marco},
   author={Christodoulou, Argyrios},
   title={Random iteration on hyperbolic Riemann surfaces},
   journal={Ann. Mat. Pura Appl. (4)},
   volume={201},
   date={2022},
   number={4},
   pages={2021--2035}
}

\bib{AbSh2025}{article}{
	author={Abate, Marco},
   author={Short, Ian},
   title={Iterated function systems of holomorphic maps},
   journal={Adv. Math.},
   volume={490},
   date={2026},
   pages={Paper No. 110818}
}

\bib{BEFRS1}{article}{
   author={Benini, Anna Miriam},
   author={Evdoridou, Vasiliki},
   author={Fagella, N\'uria},
   author={Rippon, Philip J.},
   author={Stallard, Gwyneth M.},
   title={Classifying simply connected wandering domains},
   journal={Math. Ann.},
   volume={383},
   date={2022},
   number={3-4},
   pages={1127--1178}
}

\bib{BEFRS2}{article}{
   author={Benini, Anna Miriam},
   author={Evdoridou, Vasiliki},
   author={Fagella, N\'uria},
   author={Rippon, Philip J.},
   author={Stallard, Gwyneth M.},
   title={Boundary dynamics for holomorphic sequences, non-autonomous
   dynamical systems and wandering domains},
   journal={Adv. Math.},
   volume={446},
   date={2024},
   pages={Paper No. 109673, 51}
}


\bib{BoMaSh2005}{article}{
   author={Bourdon, P. S.},
   author={Matache, V.},
   author={Shapiro, J. H.},
   title={On convergence to the Denjoy-Wolff point},
   journal={Illinois J. Math.},
   volume={49},
   date={2005},
   number={2},
   pages={405--430}
}



\bib{ChSh2019}{article}{
   author={Christodoulou, Argyrios},
   author={Short, Ian},
   title={Stability of the Denjoy-Wolff theorem},
   journal={Ann. Fenn. Math.},
   volume={46},
   date={2021},
   number={1},
   pages={421--431}
}


\bib{DoMa1991}{article}{
	author={C. I. Doering},
	author={R. Ma\~ne},
	title={The dynamics of inner functions},
	journal={Ens. Mat. Soc. Bras. Mat.},
	volume={3},
	date={1991},
	pages={1--79}
}

\bib{Du1970}{book}{
   author={Duren, Peter L.},
   title={Theory of $H\sp{p}$ spaces},
   series={Pure and Applied Mathematics},
   volume={Vol. 38},
   publisher={Academic Press, New York-London},
   date={1970}
}

\bib{Fe2023}{article}{
   author={Ferreira, Gustavo R.},
   title={A note on forward iteration of inner functions},
   journal={Bull. Lond. Math. Soc.},
   volume={55},
   date={2023},
   number={3},
   pages={1143--1153}
}

\bib{FeNi2025}{article}{
   author={Ferreira, Gustavo R.},
   author={Nicolau, Artur},
   title={Mixing and ergodicity of compositions of inner functions},
   journal={Discrete Contin. Dyn. Syst.},
   volume={45},
   date={2025},
   number={7},
   pages={2066--2080}
}


\bib{Po2010}{article}{
   author={Poggi-Corradini, Pietro},
   title={Pointwise convergence on the boundary in the Denjoy-Wolff theorem},
   journal={Rocky Mountain J. Math.},
   volume={40},
   date={2010},
   number={4},
   pages={1275--1288}
}

\bib{Po1992}{book}{
   author={Pommerenke, Ch.},
   title={Boundary behaviour of conformal maps},
   series={Grundlehren der mathematischen Wissenschaften [Fundamental
   Principles of Mathematical Sciences]},
   volume={299},
   publisher={Springer-Verlag, Berlin},
   date={1992},
   pages={x+300}
}

\bib{Sh1993}{book}{
   author={Shapiro, Joel H.},
   title={Composition operators and classical function theory},
   series={Universitext: Tracts in Mathematics},
   publisher={Springer-Verlag, New York},
   date={1993},
   pages={xvi+223}
}

\bib{Sh2000}{article}{
   author={Shapiro, Joel H.},
   title={What do composition operators know about inner functions?},
   journal={Monatsh. Math.},
   volume={130},
   date={2000},
   number={1},
   pages={57--70}
}

\bib{schiff}{book}{
   author={Schiff, Joel L.},
   title={Normal families},
   series={Universitext},
   publisher={Springer-Verlag, New York},
   date={1993},
   pages={xii+236}
}

\end{biblist}
\end{bibdiv}
\end{document}